\documentclass{amsart}

\usepackage{latexsym, amssymb, amsmath}

\newtheorem{theorem}{Theorem}[section]
\newtheorem{lemma}[theorem]{Lemma}
\newtheorem{corollary}[theorem]{Corollary}
\newtheorem{proposition}[theorem]{Proposition}
\newtheorem{remark}[theorem]{Remark}
\theoremstyle{definition}

\makeatletter
    
    \@addtoreset{equation}{section}
  \makeatother

\newcommand{\Ric}{{\rm Ric}}
\newcommand{\grad}{{\rm grad}\,}

\newcommand{\n}{\nabla}

\begin{document}
\frenchspacing
\emergencystretch=3em
\hbadness=10000
\vbadness=10000

\title[Rotational symmetry of gradient Einstein-type manifolds]
{Rotational symmetry of complete gradient Einstein-type manifolds with $\alpha=0$ under vanishing conditions on several tensor fields}

\author{Shun Maeta}
\address{Department of Mathematics,
Chiba University, 1-33, Yayoicho, Inage, Chiba, 263-8522, Japan.}
\email{shun.maeta@faculty.gs.chiba-u.jp}


\date{}

\dedicatory{}


\commby{}

\begin{abstract}
To provide a unified framework for various geometric solitons, Catino, Mastrolia, Monticelli, and Rigoli introduced gradient Einstein-type manifolds, which include gradient Yamabe solitons, gradient almost Yamabe solitons, gradient $k$-Yamabe solitons and conformal gradient solitons. 
In this paper, we show that any nontrivial complete four-dimensional gradient Einstein-type manifold with $\alpha=0$, nonnegative scalar curvature, and vanishing Bach tensor is rotationally symmetric. The Bach-flat condition arises naturally in four-dimensional geometry and physics. 
We prove the main theorem by exploiting the conformal invariance of Bach-flatness and establishing a relation between two functionals. We also provide classification results under vanishing conditions on the Weyl tensor, the Cotton tensor, and the Cao-Chen-type tensor. 
\end{abstract}

\maketitle

\section{Introduction}\label{intro} 

Geometric solitons play a fundamental role in the analysis of geometric flows. To provide a unified viewpoint for Ricci solitons, Yamabe solitons, and their generalizations, Catino, Mastrolia, Monticelli, and Rigoli \cite{CMMR17} introduced Einstein-type manifolds as part of an ambitious project.
For a smooth function $f$ on a Riemannian manifold $(M^n,g)$, a gradient Einstein-type manifold is defined by
\[
\alpha \mathrm{\Ric}+\beta \nabla\nabla f+\mu\, df\otimes df=(\varrho R+\lambda)g,
\]
where $\alpha,\beta,\mu,\varrho\in\mathbb{R}$ with $(\alpha,\beta,\mu)\neq(0,0,0)$ and $\lambda\in C^\infty(M)$.
Here $\Ric$ is the Ricci tensor and $R$ is the scalar curvature on $M$.
If $f$ is constant, then the gradient Einstein-type manifold is called trivial.

A fundamental special case is the gradient Yamabe soliton  
\begin{equation}\label{YS}
\nabla\nabla F=(R-\lambda)g,
\end{equation}
where $F$ is a smooth function on $M$ and $\lambda\in\mathbb{R}$.
Gradient Yamabe solitons are self-similar solutions to the Yamabe flow introduced by Hamilton \cite{Hamilton89}, and they are expected to control singularity formation and asymptotic profiles of the flow.
In a pioneering study of Yamabe solitons, Di Cerbo and Disconzi \cite{CD08}
proved the nonexistence of nontrivial compact Yamabe solitons and breathers. They also used the determinant of the Laplacian to develop an approach to the uniformization theorem for Riemannian surfaces.
In the compact case, Hsu \cite{Hsu12} provided a remarkably concise and clear proof of nonexistence of nontrivial gradient Yamabe solitons. Hsu \cite{Hsu14} also obtained precise asymptotic results for the conformal factors and curvatures of rotationally symmetric gradient Yamabe solitons under suitable parameter assumptions. Furthermore, nontrivial nonflat gradient Yamabe solitons were completely classified in dimension two from the point of view of Ricci solitons by Bernstein and Mettler \cite{BM15}.

Equation \eqref{YS} is the scalar-curvature analogue of the gradient Ricci soliton equation.
For gradient Ricci solitons, 
Brendle \cite{Brendle} proved Perelman's conjecture \cite{Perelman1}, namely, that three-dimensional complete noncompact $\kappa$-noncollapsed gradient steady Ricci solitons with positive sectional curvature are rotationally symmetric.
This motivates the Yamabe-soliton version of the same rigidity question, namely, the problem of determining natural geometric assumptions under which complete nontrivial gradient Yamabe solitons, or their Yamabe-type generalizations, must be rotationally symmetric.
In this paper, following Bryant's terminology \cite{Bryant05}, we allow spherical space-form quotients and call a warped product rotationally symmetric if, after a possible rescaling of the warping function, the lifted metric on its universal cover is locally of the form
$
dt^2+\psi(t)^2 g_{\mathbb S^{n-1}},
$
where $g_{\mathbb S^{n-1}}$ denotes the standard round metric. 
In particular, complete flat solitons are included in this terminology, since their universal cover is Euclidean.
In this direction, Daskalopoulos and Sesum \cite{DS13} proved rotational symmetry for complete locally conformally flat gradient Yamabe solitons with positive sectional curvature.
Soon after that, Cao, Sun, and Zhang \cite{CSZ12} showed that any complete, nontrivial, nonflat, locally conformally flat gradient Yamabe soliton with $R\geq0$ is rotationally symmetric.
From a broader viewpoint, Catino, Mantegazza, and Mazzieri \cite{CMM12} studied conformal gradient solitons and proved a global structure theorem under nonnegative Ricci curvature.

The purpose of this paper is to give a unified treatment of such Yamabe-type classification results within the framework of gradient Einstein-type manifolds with $\alpha=0$.
Gradient Einstein-type manifolds with $\alpha=0$ include gradient Yamabe solitons, gradient almost Yamabe solitons, gradient $k$-Yamabe solitons, gradient $h$-almost Yamabe solitons, conformal gradient solitons, and quasi-Yamabe gradient solitons (see \cite{BB13}, \cite{CL22}, \cite{CMM12}, \cite{HL14}, and \cite{Zeng20}).

When $\alpha\neq0$, many results have been obtained under analogous curvature assumptions, such as local conformal flatness and related conditions
\cite{BBGG11,BGG13a,BGV16,CC13,FG14}.

Previous rotational-symmetry results for gradient Yamabe solitons have been obtained under local conformal flatness and additional curvature assumptions.
In this direction, the Bach tensor 
\[
B_{ij}=\frac{1}{n-3}\nabla^k\nabla^\ell W_{ikj\ell}+\frac{1}{n-2}R^{k\ell}W_{ikj\ell}\qquad(n\geq4)
\]
 plays an important role in four-dimensional conformal geometry and physics. 
For example, on a closed $4$-dimensional Riemannian manifold $(M^4,g)$, Bach-flat metrics are precisely the critical points of the conformally invariant Weyl energy functional 
\[
\mathcal{W}(g)=\int_M|W|^2dv_g,
\]
where $W$ is the Weyl tensor of $(M,g)$ (cf. pages 133--135 in \cite{Besse87}). 
Thus, the Bach-flat condition arises naturally from a geometric variational problem.
We remark that local conformal flatness implies Bach-flatness.
Every 4-dimensional metric conformal to an Einstein metric is Bach-flat, although the converse does not hold in general (cf. \cite{GH05} and \cite{LLPV13}).
The Bach tensor also appears in 4-dimensional conformal gravity (cf. \cite{LLPV13}).

Therefore, we study nontrivial complete four-dimensional Bach-flat gradient Einstein-type manifolds with $\alpha=0$ and establish rotational symmetry under the assumption of nonnegative scalar curvature.

\begin{theorem}\label{mainB}
Any nontrivial $4$-dimensional Bach-flat complete gradient Einstein-type manifold with $\alpha=0$ and $R\geq0$ is rotationally symmetric. 
\end{theorem}

The paper is organized as follows.
In Section~\ref{Proof of main}, we provide the warped product framework and curvature formulas for complete gradient Einstein-type manifolds with $\alpha=0$ in a form adapted to the later tensorial arguments.
In Section~\ref{LCFGCS}, we classify complete gradient Einstein-type manifolds with $\alpha=0$ and $C\equiv0$, where $C$ is the Cotton tensor. Applying this result, we classify locally conformally flat gradient Einstein-type manifolds with $\alpha=0$. To prove rotational symmetry of complete gradient Einstein-type manifolds with $\alpha=0$, we consider a warped product. Since Bach-flatness is preserved under conformal change in dimension four, we consider two functionals $\mathcal{W}$ and 
\[
\mathcal{E}(\bar g)=\int_N \left|\overline {\Ric}-\frac{\bar R}{3}\bar g\right|^2dv_{\bar g},
\] 
for a Riemannian manifold $(N^3,\bar g)$, where $\overline\Ric$ and $\bar R$ are the Ricci tensor and the scalar curvature of $N$ (cf. \cite{Anderson01},  see also \cite{BD26}).
Combining these with the assumption $R\geq0$, we show that Bach-flatness yields local conformal flatness.
By the results shown above, we obtain that $(M,g)$ is rotationally symmetric.
In Section~\ref{vanish}, we also classify complete gradient Einstein-type manifolds with $\alpha=0$ under the assumption that the Cao-Chen-type tensor $D^F$ vanishes.

Throughout this paper, unless otherwise stated, we assume that $n\geq3$.


\section{Warped product structure and curvature formulas}\label{Proof of main}
This section provides the warped product framework and curvature identities used in the proofs of the main classification results.
We first fix the curvature convention used throughout the paper.
The Riemannian curvature tensor is defined by 
\[
R(X,Y)Z=-\n_X\n_YZ+\n_Y\n_XZ+\n_{[X,Y]}Z
\]
for $X,Y,Z\in \mathfrak{X}(M)$. The Ricci tensor $R_{ij}$ is defined by 
$R_{ij}=g^{k\ell}R_{ikj\ell}$, where $R_{ijk\ell}=g(R(\partial_i,\partial_j)\partial_k,\partial_\ell)$.

For gradient Einstein-type manifolds with $\alpha=0$, we have
\[
\beta \nabla\nabla f+\mu\, df\otimes df=(\varrho R+\lambda)g.
\]
By an elementary argument, if $\beta=0$, then complete gradient Einstein-type manifolds with $\alpha=0$ are trivial.

If $\beta\not=0$ and $\mu\not=0$ and we set $F=\frac{\beta}{\mu}e^{(\mu/\beta)f}$, then we have
\[
\nabla\nabla F=\frac{\mu}{\beta^2}(\varrho R+\lambda)F\, g.
\]
Hence, if we take $\varphi=\frac{\mu}{\beta^2}(\varrho R+\lambda)F$, the resulting equation is of the form 
\begin{equation}\label{GCS}
\nabla\nabla F=\varphi g.
\end{equation} 
When $\beta\not=0$ and $\mu=0$, if we take $F=\beta f$ and $\varphi= \varrho R+\lambda$, then we have $\eqref{GCS}$. 
In both cases, $F$ is constant if and only if $f$ is constant.

Therefore, in the following we only need to consider \eqref{GCS} to study nontrivial complete gradient Einstein-type manifolds with $\alpha=0$. 
Hereafter, when we write $(M^n,g,F,\varphi)$ for a nontrivial gradient Einstein-type manifold with $\alpha=0$, the functions $F$ and $\varphi$ are those defined above.
 
The geometric structure generated by a nonconstant solution of \eqref{GCS} is also classical in the theory of concircular fields, going back to Yano \cite{Yano40} and Tashiro \cite{Tashiro65}, and it appears in the warped product rigidity theory of Cheeger and Colding \cite{CC96}. Concircular field equations were treated by Sharma and Deshmukh \cite{SD21} 
(see also \cite{Xu93}, \cite{KR95}).

The following lemma records the form of Tashiro's classification theorem needed in this paper. In \cite{CMMR17}, Catino, Mastrolia, Monticelli, and Rigoli also provided the structure theorem for Einstein-type manifolds with $\alpha=0$ (Theorem 1.4 in \cite{CMMR17}).

\begin{lemma}[\cite{CMM12}, \cite{CMMR17}, \cite{Tashiro65}]\label{main}
Let $(M^n,g)$ be a complete Riemannian manifold, and let $F$ be a nonconstant smooth function on $M$ satisfying $\nabla\nabla F=\varphi g$ for some smooth function $\varphi$. Then one of the following holds.

$(1)$ $(M,g)$ is compact and rotationally symmetric, or

$(2)$ $(M,g)$ is rotationally symmetric and is given by the warped product
$$([0,\infty),dr^2)\times_{|\nabla F|}(\mathbb{S}^{n-1},{\bar g}_{S}),$$
where $\bar g_{S}$ is the round metric on $\mathbb{S}^{n-1},$ or

$(3)$ $(M,g)$ is the warped product 
$$(\mathbb{R},dr^2)\times_{|\nabla F|} \left(N^{n-1},\bar g\right).$$
\end{lemma}

 In case $(3)$, we write
\[
(M,g)=(\mathbb R,dr^2)\times_{F'(r)}(N^{n-1},\bar g),
\]
where $F'(r)=|\nabla F|>0$ on $\mathbb R$. The standard curvature formulas \cite{ONiell} for this warped product give, for $a,b,c,d=2,3,\cdots,n$,
\begin{align}\label{case3curv}
R_{1a1b}&=-F'F'''\bar g_{ab},\quad R_{1abc}=0,\\
R_{abcd}&=(F')^2\bar R_{abcd}+(F'F'')^2(\bar g_{ad}\bar g_{bc}-\bar g_{ac}\bar g_{bd}),\notag
\end{align}
\begin{align}\label{case3ric}
R_{11}&=-(n-1)\frac{F'''}{F'},\quad R_{1a}=0,\\
R_{ab}&=\bar R_{ab}-((n-2)(F'')^2+F'F''')\bar g_{ab}.\notag
\end{align}
Moreover,
\begin{align}\label{case3scalar}
R=(F')^{-2}\bar R-(n-1)(n-2)\Big(\frac{F''}{F'}\Big)^2-2(n-1)\frac{F'''}{F'}.
\end{align}
Since $F''=\varphi$ and $F'''=\frac{1}{|\nabla F|}g(\nabla F,\nabla\varphi)$, \eqref{case3scalar} is equivalent to
\[
|\nabla F|^2R
=
\bar R
-(n-1)(n-2)\varphi^2
-2(n-1)g(\nabla F,\nabla\varphi).
\]

In case (3) of Lemma \ref{main}, we can choose the direction of $r$ such that $F'(r)>0$. In the following, set $\rho(r)=F'(r)$.

\begin{remark}
In Lemma~\ref{main}, cases $(1)$ and $(2)$ occur when $F$ has critical points, while case $(3)$ occurs when $F$ has no critical points. See, for example, \cite{Tashiro65} and \cite{Petersen16}. Thus the only case in which rotational symmetry is not immediate from the structure of the level sets is case $(3)$. The curvature assumptions imposed later in this paper are used precisely to restrict the fiber $(N^{n-1},\bar g)$ in this case.
\end{remark}

\begin{remark}
We record how Lemma~\ref{main} specializes to the three-dimensional gradient Yamabe soliton case studied in \cite{CSZ12}.
Indeed, by the nonexistence theorem for compact nontrivial gradient Yamabe solitons (cf. \cite{Hsu12}), case $(1)$ of Lemma~\ref{main} cannot occur.
Consider case $(3)$ of Lemma~\ref{main}.
By the soliton equation $R-\lambda=\varphi=F''$ and
$$R=(F')^{-2}\bar R-2\Big(\frac{F''}{F'}\Big)^2-4\frac{F'''}{F'},$$
the scalar curvature $\bar R$ is constant. Therefore, $N^2$ is a space form. In particular, a three-dimensional complete nontrivial gradient Yamabe soliton with $R\geq0$ is rotationally symmetric. Thus, the present formulation contains the usual Yamabe-soliton rotational symmetry result as its first nontrivial special case.
\end{remark}


\section{Proof of Theorem \ref{mainB}}\label{LCFGCS}

We begin with a positivity lemma for the scalar curvature $\bar R$ of $N$ in case (3) of Lemma \ref{main} under $R\geq0$. 

\begin{lemma}\label{nonnegatives}
In case (3) of Lemma \ref{main}, if $R\geq0$, then $\bar R\geq0$. Moreover, if $\varphi\not\equiv0$, then $\bar R>0$.
\end{lemma}

\begin{proof}
By \eqref{case3scalar}, if there exists $p_0\in N$ such that $\bar R(p_0)<0$, then for any $r\in\mathbb{R}$, we have
\begin{align*}
0&\leq R(r,p_0)F'^2(r)\\
&=\bar R(p_0)-(n-1)(n-2)(F''(r))^2-2(n-1)F'(r)F'''(r)\\
&\leq-2(n-1)F'(r)F'''(r).
\end{align*}
Hence $F'''(r)\leq0$. Since the positive function $F'(r)$ is concave on $\mathbb{R}$, $F'(r)$ is constant. By \eqref{case3scalar} again, we have 
\[
0\leq R(r,p_0)F'^2(r)=\bar R(p_0)<0,
\]
which is a contradiction.
Hence $\bar R\geq0$. 
Since $\varphi\not\equiv0$ implies that $F'$ is nonconstant, a similar argument shows that $\bar R>0$. 
\end{proof}

\begin{lemma}\label{thmA}
Any complete, nontrivial, locally conformally flat gradient Einstein-type manifold with $\alpha=0$ and $R\geq0$ is rotationally symmetric.
\end{lemma}

\begin{remark}
Let us point out what Lemma~\ref{thmA} suggests about the Yamabe-soliton problem.
Lemma~\ref{thmA} is proved for gradient Einstein-type manifolds with $\alpha=0$, a class much broader than gradient Yamabe solitons.
Thus the assumption of local conformal flatness, even when combined with $R\geq0$, appears to be too strong to isolate the specific difficulty of proving rotational symmetry for Yamabe solitons.
In particular, if one seeks a rotational-symmetry theorem for general complete gradient Yamabe solitons under bounded scalar curvature or positive scalar curvature, the present result suggests that merely weakening the assumption of local conformal flatness may not be the decisive approach.
The essential difficulty is expected to come from the Yamabe-soliton equation itself and from scalar-curvature information.
\end{remark}

We begin by fixing the conventions for the Cotton tensor $C$ and the Weyl tensor $W$. 
\begin{align*}
C_{ijk}
=&\nabla_iS_{jk}-\nabla _jS_{ik}\\
=&\nabla_iR_{jk}-\nabla_jR_{ik}-\frac{1}{2(n-1)}(g_{jk}\nabla_iR-g_{ik}\nabla_jR),
\end{align*}
where $S=\Ric-\frac{1}{2(n-1)}Rg$ is the Schouten tensor.
The Cotton tensor is skew-symmetric in the first two indices and totally trace-free, that is,
$$C_{ijk}=-C_{jik} \quad \text{and} \quad g^{jk}C_{ijk}=g^{ik}C_{ijk}=0.$$
\begin{align}\label{Weyl}
W_{ijk\ell}
=&R_{ijk\ell}-\frac{1}{n-2}(R_{ik}g_{j\ell}+R_{j\ell}g_{ik}-R_{i\ell}g_{jk}-R_{jk}g_{i\ell})\\
&+\frac{R}{(n-1)(n-2)}(g_{ik}g_{j\ell}-g_{i\ell}g_{jk}).\notag
\end{align}
We recall the following standard tensorial characterization of local conformal flatness. A Riemannian manifold $(M^n,g)$ is locally conformally flat if and only if either the Weyl tensor vanishes for $n\geq4$, or the Cotton tensor vanishes for $n=3$.
Moreover, for $n\geq4$, the vanishing of the Weyl tensor implies the vanishing of the Cotton tensor. 
For $n=3$, the Weyl tensor always vanishes, whereas the Cotton tensor does not vanish in general.

First, we consider gradient Einstein-type manifolds with $\alpha=0$ under the weaker assumption that the Cotton tensor vanishes. This will be used to prove Lemma~\ref{thmA}.

\begin{proposition}\label{subC}
Let $(M^n,g,F,\varphi)$ be a nontrivial complete gradient Einstein-type manifold with $\alpha=0$ and $C\equiv0$. Then one of the following holds.

$(1)$ $(M,g)$ is compact and rotationally symmetric, or

$(2)$ $(M,g)$ is rotationally symmetric and is given by the warped product
$$([0,\infty),dr^2)\times_{|\nabla F|}(\mathbb{S}^{n-1},{\bar g}_{S}),$$
where $\bar g_{S}$ is the round metric on $\mathbb{S}^{n-1},$ or

$(3)$ $(M,g)$ is the warped product 
$$(\mathbb{R},dr^2)\times_{|\nabla F|} \left(N^{n-1},\bar g\right),$$
where $N$ has constant scalar curvature $\bar R$.
Moreover, if $\varphi\not\equiv0$, then $N$ is Einstein.
Furthermore, if $R\geq0$, then either $\bar R>0$, or $R=\bar R=0$ and $(M,g)$ is isometric to the Riemannian product $(\mathbb{R},dr^2)\times(N^{n-1},\bar g)$ after a constant rescaling of the fiber metric. Moreover, if $R\geq0$ and $\varphi\not\equiv0$, then $\bar R>0$.
\end{proposition}

\begin{remark}
In case (3) of Proposition \ref{subC}, $(\mathbb{R}\times S^2\times S^1,g=dr^2+\bar g_{S^2}+\bar g_{S^1},F(r)=r,\varphi=0)$ satisfies $\nabla\nabla F=0$ with $C\equiv0$, but $N=S^2\times S^1$ is not Einstein.
Hence the assumption $\varphi\not\equiv0$ is needed.
\end{remark}

\begin{proof}
The proof follows the approach of Cao, Sun and Zhang \cite{CSZ12}. For completeness, we provide the details in our more general setting.

By Lemma~\ref{main}, cases $(1)$ and $(2)$ are already rotationally symmetric, so it remains to analyze case $(3)$.
 In this case, we use the curvature formulas \eqref{case3ric} and \eqref{case3scalar}.
We have $g=dr^2+\rho^2(r)\bar g$.

By the definition of the Cotton tensor, we have  
\begin{align*}
C_{1a1}
=&\nabla_1R_{a1}-\nabla_aR_{11}-\frac{1}{2(n-1)}(g_{a1}\nabla_1R-g_{11}\nabla_aR)\\
=&\frac{1}{2(n-1)}\nabla_a R.
\end{align*}
From this and \eqref{case3scalar}, the scalar curvature $\bar R$ of $N$ is constant. 

We calculate 
\begin{equation}\label{c1ab}
C_{1ab}=\nabla_1R_{ab}-\nabla_aR_{1b}-\frac{1}{2(n-1)}
\left(
g_{ab}\nabla_1 R-g_{1b}\nabla_aR\right).
\end{equation}
Since
\begin{align*}
\Gamma_{1a}^k
&=\frac{1}{2}g^{k\ell}(\partial_1g_{\ell a}+\partial_ag_{1\ell}-\partial_{\ell}g_{1a})\\
&=\frac{1}{2}g^{kc}\partial_{1}g_{ca}\\
&=\frac{\rho'}{\rho}\delta^k_{a},
\end{align*}
we obtain
\begin{equation}\label{r1ab}
\nabla_1R_{ab}=\partial_r\Ric_{ab}-2\frac{\rho'}{\rho}R_{ab}.
\end{equation}
Since 
\begin{align*}
\Gamma_{ab}^1=-\rho\rho'\overline g_{ab},
\end{align*}
we have
\begin{equation}\label{ra1b}
\nabla_{a}R_{1b}=-\frac{\rho'}{\rho}R_{ab}+\rho\rho'R_{11}\overline g_{ab}.
\end{equation}
Substituting \eqref{case3ric}, \eqref{case3scalar}, \eqref{r1ab}, and \eqref{ra1b} into \eqref{c1ab}, we obtain
\begin{align*}
C_{1ab}
&=-\frac{\rho'}{\rho}\overline R_{ab}\\
&\qquad+\Big\{
(-n+3)\rho'\rho''-\rho\rho'''
+(n-2)\rho'^3\rho^{-1}-\frac{R'}{2(n-1)}\rho^2
\Big\}\bar g_{ab}\\
&=-\frac{\rho'}{\rho}
\left(
\bar R_{ab}-\frac{\bar R}{n-1}\bar g_{ab}
\right).
\end{align*}
If the Cotton tensor vanishes and $\rho$ is nonconstant, 
\[
\bar R_{ab}=\frac{\bar R}{n-1}\bar g_{ab}.
\]
Hence $N$ is Einstein.

Assume that $\rho$ is constant.
By \eqref{case3scalar}, $R$ is also constant. 

In both cases, $\bar R$ is constant. 
If $ R\geq0$, then by Lemma~\ref{nonnegatives}, $\bar R\geq0$. If $\rho$ is nonconstant, then $\bar R>0$. If $\rho$ is constant, then $\bar R>0$ or $R=\bar R=0$ and $(M,g)$ is the Riemannian product $(\mathbb{R},dr^2)\times(N^{n-1},\bar g)$ after a constant rescaling of the fiber metric.
\end{proof}

The next lemma shows that local conformal flatness strengthens the fiber geometry from constant scalar curvature to constant sectional curvature in case $(3)$ in Lemma \ref{main}.

\begin{lemma}\label{nocriticalpoint}
Let $(M,g,F,\varphi)$ be a nontrivial complete locally conformally flat gradient Einstein-type manifold with $\alpha=0$. Assume that $F$ has no critical points. Then $(M,g)$ is a warped product
$$(\mathbb{R},dr^2)\times_{|\nabla F|} \left(N^{n-1}(c),\bar g\right),$$ 
where $(N^{n-1}(c),\bar g)$ is a space form.
\end{lemma}

\begin{proof}
We consider case $(3)$ of Lemma~\ref{main}. We use the curvature formulas \eqref{case3curv}, \eqref{case3ric}, and \eqref{case3scalar}.

If $\dim M=3$, then Proposition~\ref{subC} shows that $N$ has constant scalar curvature. Since $N$ is two-dimensional, it is a space form. 

Assume next that $\dim M \geq4$.
By \eqref{Weyl}, \eqref{case3curv}, \eqref{case3ric} and \eqref{case3scalar}, we have
\begin{align*}
W_{1a1b}=&-\frac{\bar R_{ab}}{n-2}+\frac{\bar R}{(n-1)(n-2)}\bar g_{ab},\\
W_{1abc}=&0,\\
W_{abcd}
=&(F')^2\Big(\bar W_{abcd}\\
&+\frac{1}{(n-2)(n-3)}
\Big\{\frac{2}{n-1}\bar R(\bar g_{ad}\bar g_{bc}-\bar g_{ac}\bar g_{bd})\\
&\hspace{100pt}
-(\bar R_{ad}\bar g_{bc}+\bar R_{bc}\bar g_{ad}-\bar R_{ac}\bar g_{bd}-\bar R_{bd}\bar g_{ac})\Big\}\Big).
\end{align*}
Since $M$ is locally conformally flat, we have
\begin{equation}\label{NisEin}
\bar R_{ab}=\frac{\bar R}{n-1}\bar g_{ab},
\end{equation}
and 
\begin{align}\label{NisLCF}
\bar W_{abcd}
=&-\frac{1}{(n-2)(n-3)}
\Big\{\frac{2}{n-1}\bar R(\bar g_{ad}\bar g_{bc}-\bar g_{ac}\bar g_{bd})\\
&\hspace{70pt}-(\bar R_{ad}\bar g_{bc}+\bar R_{bc}\bar g_{ad}-\bar R_{ac}\bar g_{bd}-\bar R_{bd}\bar g_{ac})\Big\}.\notag
\end{align}
If $n=4$, then $\dim N=3$, and \eqref{NisEin} already implies that $N$ has constant sectional curvature.
If $n\geq5$, then substituting \eqref{NisEin} into \eqref{NisLCF} gives $\bar W_{abcd}=0$, so $N$ is Einstein and locally conformally flat. Hence $N$ has constant sectional curvature.
Therefore, $N$ is a space form.
\end{proof}

Combining Lemma~\ref{nocriticalpoint} with Lemma~\ref{main}, we obtain the following.

\begin{lemma}\label{sub2}
Let $(M^n,g,F,\varphi)$ be a nontrivial complete locally conformally flat gradient Einstein-type manifold with $\alpha=0$. Then one of the following holds.

$(1)$ $(M,g)$ is compact and rotationally symmetric, or

$(2)$ $(M,g)$ is rotationally symmetric and is given by the warped product
$$([0,\infty),dr^2)\times_{|\nabla F|}(\mathbb{S}^{n-1},{\bar g}_{S}),$$
where $\bar g_{S}$ is the round metric on $\mathbb{S}^{n-1},$ or

$(3)$ $(M,g)$ is the warped product 
$$(\mathbb{R},dr^2)\times_{|\nabla F|} \left(N^{n-1}(c),\bar g\right),$$
where $(N^{n-1}(c),\bar g)$ is a space form.
\end{lemma}

\begin{proof}[Proof of Lemma~\ref{thmA}]
By Lemma~\ref{sub2}, it remains to consider case $(3)$.
By Lemma~\ref{nonnegatives}, we have $\bar R\geq0$.
If $\bar R>0$, then the fiber has positive constant sectional curvature, and hence $(M,g)$ is rotationally symmetric.
If $\bar R=0$, then Lemma~\ref{nonnegatives} implies that $F'$ is constant. 
Since the fiber is flat, $(M,g)$ is also flat. Its universal cover is Euclidean space by completeness. Therefore, $(M,g)$ is rotationally symmetric in this case as well.
\end{proof}

We now prove Theorem \ref{mainB} by using Lemma \ref{thmA}.

\begin{proof}
We only need to consider case (3) of Lemma \ref{main}. 
By Lemma \ref{nonnegatives}, $\bar R\geq0$. 
Set 
\[
ds=\frac{1}{\rho(r)}dr.
\]
Then the metric is 
\[
g=dr^2+\rho(r)^2\bar g=\rho^2(ds^2+\bar g).
\] 
In dimension 4, Bach-flatness is preserved under conformal change. Hence
$\hat g=ds^2+\bar g$ is also Bach-flat.
We consider $\tilde g=d\theta^2+\bar g$ on $S^1\times N$, where $d\theta^2$ is the standard metric on the circle $S^1=\mathbb{R}/\mathbb{Z}$ of length one.  Since $\hat g$ and $\tilde g$ are locally isometric and $\hat g$ is Bach-flat, $\tilde g$ is also Bach-flat. For a smooth compactly supported symmetric $(0,2)$-tensor $h$ on $N$ and sufficiently small $t$, set $\bar g_t=\bar g+th$. 
Let $U\subset N$ be a compact subset with smooth boundary such that ${\rm supp}~h\subset {\rm Int}\,U$. 
The Bach tensor can be realized as $-\frac{1}{4}\grad (\mathcal{W})$ (cf. page 3 in \cite{Helliwell20}).
In the following, we consider $\mathcal{W}$ on $S^1\times U$, that is, 
\[
\mathcal{W}_{{S^1}\times U}(\tilde g)=\int_{{S^1}\times U}|W|^2dv_{\tilde g},
\]
 and $\mathcal{E}$ on $U$, that is, 
\[
\mathcal{E}_U(\bar g)=\int_{U}\left|
\overline \Ric -\frac{\bar R}{3}\bar g\right|^2dv_{\bar g}.
\]
For the metric $d\theta^2+\bar g_t$, by the definition of the Weyl tensor, we have
\[
\mathcal{W}_{S^1\times U}(d\theta^2+\bar g_t)
=2\mathcal{E}_U(\bar g_t).
\]
Combining this with Bach-flatness, we have
\[
0=\left.\frac{d}{dt}\right|_{t=0}\mathcal{W}_{S^1\times U}(d\theta^2+\bar g_{t})
=2\left.\frac{d}{dt}\right|_{t=0}\mathcal{E}_{U}(\bar g_t).
\]
Since $h$ is arbitrary, $\bar g$ also satisfies the Euler-Lagrange equation of $\mathcal{E}$.
The Euler-Lagrange equation for $\mathcal{E}$ is shown by Anderson (cf. page 53 in \cite{Anderson01}). The following equation holds
\[
\overline\Delta \bar R+3\left|
\overline \Ric -\frac{\bar R}{3}\bar g\right|^2
=0,
\]
where $\overline\Delta$ is the Laplace-Beltrami operator of $\bar g$.
We remark that the fiber $(N,\bar g)$ is complete because $(M,g)$ is complete. 

If $N$ is compact, then by the standard maximum principle, we have
\[
\overline\Ric =\frac{\bar R}{3}\bar g.
\]
Hence $N$ is Einstein. Since the dimension of $N$ is three, $N$ has constant sectional curvature.

If $N$ is noncompact, since $\bar R\geq0$, Proposition 5.4 in \cite{Anderson01} yields that $(N,\bar g)$ is flat. In both cases, the sectional curvature is constant. Hence, the Weyl tensor vanishes. By Lemma \ref{thmA}, $(M,g)$ is rotationally symmetric.
\end{proof}

\section{Gradient Einstein-type manifolds with $\alpha=0$ and a vanishing condition on the Cao-Chen-type tensor}\label{vanish}

For a gradient Ricci soliton satisfying $\Ric-\lambda g=\nabla\nabla F$, Cao and Chen introduced a new tensor $D$ (cf.~\cite{CC13}, \cite{CC12}). In the present setting, we use the same algebraic expression with the potential function $F$ and call it the {\it Cao--Chen-type tensor associated with $F$}. We denote it by $D^F$. For $n\geq3$, 
\begin{align*}
D^F_{ijk}
=&\frac{1}{n-2}(R_{kj}\n_i F-R_{ki}\n_jF)\\
&+\frac{1}{(n-1)(n-2)} (R_{it}g_{jk}\n^t F-R_{jt}g_{ik}\n^t F)\\
&-\frac{R}{(n-1)(n-2)}(g_{kj}\n_i F-g_{ki}\n_j F).
\end{align*}
For gradient Ricci solitons, this tensor appears in the identity relating the Cotton tensor to the Weyl tensor contracted with the potential vector field.
In the setting of gradient Einstein-type manifolds with $\alpha=0$, the defining equation does not explicitly contain the Ricci tensor, so the corresponding identity for gradient Ricci solitons does not directly apply. The point of the following lemma is that the algebraic condition $D^F\equiv0$ nevertheless gives a strong geometric restriction in the warped product framework of Lemma~\ref{main}.

\begin{lemma}\label{main2}
Let $(M^n,g,F,\varphi)$ be a nontrivial complete gradient Einstein-type manifold with $\alpha=0$ and $D^F\equiv0$. Then one of the following holds.

$(1)$ $(M,g)$ is compact and rotationally symmetric, or

$(2)$ $(M,g)$ is rotationally symmetric and is given by the warped product
$$([0,\infty),dr^2)\times_{|\nabla F|}(\mathbb{S}^{n-1},{\bar g}_{S}),$$
where $\bar g_{S}$ is the round metric on $\mathbb{S}^{n-1},$ or

$(3)$ $(M,g)$ is the warped product 
$$(\mathbb{R},dr^2)\times_{|\nabla F|} \left(N^{n-1},\bar g\right),$$
where the fiber satisfies
\[
\bar R_{ab}=\frac{\bar R}{n-1}\bar g_{ab}.
\]
In particular, if $n\geq4$, then $N$ is an Einstein manifold. Furthermore, if $n\geq4$ and $R\geq0$, then either $\bar R>0$, or $R=\bar R=0$ and $(M,g)$ is isometric to the Riemannian product $(\mathbb{R},dr^2)\times(N^{n-1},\bar g)$, after a constant rescaling of the fiber metric, where $N$ is Ricci flat. 
\end{lemma}

\begin{proof}
By Lemma~\ref{main}, the cases with critical points are already rotationally symmetric, so it remains to analyze case $(3)$.
In this case, we use the curvature formulas \eqref{case3ric} and \eqref{case3scalar}.
By the definition of the Cao--Chen-type tensor $D^F$, we have
\begin{align*}
D^F_{1ab}
=&\frac{1}{n-2}(R_{ba}\n_1 F-R_{b1}\n_aF)\\
&+\frac{1}{(n-1)(n-2)} (R_{1t}g_{ab}\n^t F-R_{at}g_{1b}\n^t F)\\
&-\frac{R}{(n-1)(n-2)}(g_{ba}\n_1 F-g_{b1}\n_a F)\\
=&\frac{1}{n-2}R_{ba}\n_1 F
+\frac{1}{(n-1)(n-2)} R_{11}g_{ab}\n_1 F\\
&-\frac{R}{(n-1)(n-2)}g_{ba}\n_1 F.
\end{align*}
Since the Cao--Chen-type tensor vanishes and $F'>0$, we have
$$R_{ab}=\frac{R-R_{11}}{n-1}g_{ab}.$$
From this, \eqref{case3ric} and \eqref{case3scalar}, we have
$$\bar R_{ab}=\frac{\bar R}{n-1}\bar g_{ab}.$$
This is the stated fiber condition.
If $n\geq4$, then $\dim N\geq3$, and the contracted Bianchi identity on $N$ implies that $\bar R$ is constant. Hence, $N$ is an Einstein manifold.
Assume in addition that $n\geq4$ and $R\geq 0$. If $\bar R\leq0$, then by \eqref{case3scalar}, $F'''\leq0$. 
Therefore, the positive function $F'$ is concave on the whole real line.
As above, a positive concave function on $\mathbb R$ is constant, and hence $F'$ is constant.
By \eqref{case3scalar} again, $R=\bar R=0$ and $\overline{\rm Ric}=0$ on $N$. Therefore, after a constant rescaling of the fiber metric, $(M,g)$ is isometric to the Riemannian product $(\mathbb{R},dr^2)\times(N^{n-1},\bar g),$ where $N$ is Ricci flat.
\end{proof}

\begin{remark}
When $n=3$, the fiber in case $(3)$ is two-dimensional, and the condition displayed in Lemma~\ref{main2} is automatic. Thus dimension four is the first case in which the Einstein condition obtained from $D^F\equiv0$ forces the fiber to be a space form. This gives the low-dimensional consequence recorded below.
\end{remark}

If $\dim N=n-1=3$, then the Einstein manifold $N$ is a space form. Hence, we obtain the following structure theorem.

\begin{proposition}\label{lowDstructure}
Let $(M^4,g,F,\varphi)$ be a nontrivial complete gradient Einstein-type manifold with $\alpha=0$ and $D^F\equiv0$. Then one of the following holds.

$(1)$ $(M,g)$ is compact and rotationally symmetric, or

$(2)$ $(M,g)$ is rotationally symmetric and is given by the warped product
$$([0,\infty),dr^2)\times_{|\nabla F|}(\mathbb{S}^{3},{\bar g}_{S}),$$
where $\bar g_{S}$ is the round metric on $\mathbb{S}^{3},$ or

$(3)$ $(M,g)$ is the warped product 
$$(\mathbb{R},dr^2)\times_{|\nabla F|} \left(N^{3}(c),\bar g\right),$$
after a constant rescaling of the fiber metric, where $(N^{3}(c),\bar g)$ is a space form.
\end{proposition}

\begin{proof}
Lemma~\ref{main2} gives a three-dimensional Einstein fiber in case $(3)$, and every three-dimensional Einstein manifold has constant sectional curvature. Hence the fiber is a space form, which gives the stated conclusions.
\end{proof}

As a corollary, we obtain rotational symmetry of nontrivial complete four-dimensional gradient Einstein-type manifolds with $\alpha=0$, $D^F\equiv0$, and $R\geq0$.
 
\begin{corollary}\label{thmB}
Let $(M^4,g,F,\varphi)$ be a complete, nontrivial gradient Einstein-type manifold with $\alpha=0$.
Assume that $D^F\equiv0$ and $R\geq0$.
Then $(M^4,g)$ is rotationally symmetric.
\end{corollary}

\begin{proof}[Proof of Corollary~\ref{thmB}]
By Proposition~\ref{lowDstructure}, cases $(1)$ and $(2)$ are already rotationally symmetric. 
In case $(3)$, the fiber has constant sectional curvature. Hence $(M,g)$ is locally conformally flat.
The conclusion follows from Lemma~\ref{thmA}.
\end{proof}


\subsection*{Acknowledgments}
The author is partially supported by the Grant-in-Aid for Scientific Research (C), No.~23K03107, Japan Society for the Promotion of Science.

\subsection*{Statements and Declarations}
The author declares that he has no conflict of interest.

\subsection*{Data availability statement} 
Data sharing is not applicable to this article since no datasets were generated or analyzed during the current study.

\bibliographystyle{amsbook} 

\end{document}